\documentclass{article}

\usepackage[utf8]{inputenc} 
\usepackage{amsmath,amssymb,amsfonts}
\usepackage{amsthm}

\usepackage{hyperref}       
\usepackage{url}            
\usepackage{booktabs}       
\usepackage{amsfonts}       
\usepackage{nicefrac}       
\usepackage{microtype}      
\usepackage{geometry}
\usepackage{xcolor}         
\usepackage{mathtools}
\mathtoolsset{showonlyrefs}
\usepackage{comment}
\usepackage{hyperref}
\usepackage{amsbsy}
\usepackage{amscd}
\usepackage{newtxmath}
\usepackage{epic}

\newtheorem{theorem}{Theorem}
\newtheorem{lemma}[theorem]{Lemma} 
\newtheorem{proposition}[theorem]{Proposition}
\newtheorem{remark}[theorem]{Remark}
\newtheorem{corollary}[theorem]{Corollary}

\newtheorem{assumption}{Assumption}
\newtheorem{definition}{Definition}

\newcommand{\R}{\mathbb{R}}

\newcommand{\Z}{\mathbb{Z}}
\renewcommand{\d}{\mathrm{d}}
\renewcommand{\i}{\mathrm{i}}

\def\II{\mathcal{I}}

\def\cM{\mathcal{M}}
\def\cH{\mathcal{H}}

\def\cF{\mathcal{F}}
\def\cH{\mathcal{H}}

\def\NN{\mathbb{N}}

\def\RR{\mathbb{R}}

\def\supp{\text{\rm supp}}
\def\sinc{\text{\rm sinc}}

\newcommand{\norm}[1]{\left\lVert#1\right\rVert}

\providecommand{\keywords}[1]
{
  \small	
  \textbf{\textit{Keywords---}} #1
}

\providecommand{\classification}[1]
{
  \small	
  \textbf{\textit{Subject classification---}} #1
}

\begin{document}

\title{On Coorbit Fr\'echet Spaces}

\author{
S. Dahlke\footnotemark[1]
\and
F.~De~Mari\footnotemark[2]
\and
E.~De Vito\footnotemark[2]
\and
M. Hansen\footnotemark[5]
\and
G. Steidl\footnotemark[3]
\and
G. Teschke\footnotemark[4]
}

\renewcommand{\thefootnote}{\fnsymbol{footnote}}

\footnotetext[1]{FB12 Mathematik und Informatik, Philipps-Universit\"at
  Marburg, Hans-Meerwein Stra{\ss}e, Lahnberge, 35032 Marburg, Germany, email: dahlke@mathematik.uni-marburg.de} 
\footnotetext[2]{Dipartimento di Matematica, Universit\`a di Genova,  Via
  Dodecaneso 35, Genova,   Italy, email: filippo.demari@unige.it, ernesto.devito@unige.it}
\footnotetext[3]{Institute of Mathematics, TU Berlin, Str. des 17. Juni 136, Berlin, email:steidl@math.tu-berlin.de}
\footnotetext[4]{Institute for Computational Mathematics in Science and
  Technology, Hochschule Neubrandenburg, University of Applied
  Sciences, Brodaer Str. 2, 17033 Neubrandenburg, Germany}
\footnotetext[5]{Department of Mathematics,
  Friedrich-Alexander-Universit\"at Erlangen-N\"urnberg University, Cauerstra{\ss}e 11
91058 Erlangen}

\maketitle

\begin{abstract}
This paper is concerned with a new approach  to coorbit space theory.  Usually, coorbit spaces are
 defined by collecting all distributions for which  the voice transform associated with
a square-integrable group representation   possesses  a certain  decay, usually measured in a Banach space norm 
such as weighted $L_p$-norms.  Unfortunately, in cases  where  the representation does not satisfy certain integrability 
conditions,  one is  faced with a   bottleneck, namely  that the discretization of the coorbit spaces
 is surprisingly difficult.   It turns out that in these cases  the construction of  coorbit  spaces as Fr\'echet spaces is much more convenient  since then 
  atomic decompositions  can be established in a very natural way. 
\end{abstract}


\classification{41A30,46E3,46A04, 42C15 } 
\section{Introduction}
One of the most important tasks in applied mathematics is the analysis of signals. These signals might be given explicitly, e.g., in image analysis, or implicitly,  as solutions of operator equations, say. The first step is always to apply a suitable transformation. By now, an impressive amount  of different  transformations exists, such as the Fourier transform, the Gabor transform, the wavelet transform, or the shearlet transform, just to name a few. Which one to choose   clearly depends on the information one wants to extract from the signal An  important observation is that many of these transforms stem from a square-integrable group representation, e.g., the wavelet transform is associated with the affine group or $ax+b$-group. At this point, the very important coorbit theory comes into play which has been developed by Feichtinger and Gr\"ochenig in a series of papers  \cite{FeiGr0, FeiGr1, FeiGr2, Gr}.  We also refer to \cite{DDGL}, Chapter 3.2 and \cite{B} for an overview.  This theory allows for a unified treatment of the different transformations. Moreover, it provides the construction of  natural smoothness spaces, the coorbit spaces, where smoothness is measured by the decay of the voice transform associated with the group representation. In addition, by discretizing the representation,  atomic decompositions and Banach frames for the coorbit spaces can be obtained. To apply this theory,  besides the square integrability,
 a second very important condition has to be satisfied, namely,  the reproducing kernel associated with the representation has to be  contained in a weighted $L_1$-space. Unfortunately, in some natural cases, this assumption is not satisfied, see, e.g.,  Section \ref{sec:ex}.   Nevertheless, in \cite{DDDLSTV}, it has been shown that  the coorbit theory can also be generalized to this case. Then, the space of distributions from which the coorbit spaces are constructed, has to be modified.   Instead of using the dual of a Banach space, it is  more suitable  to choose the dual space of a Fr\'echet space. By proceeding this way, coorbit spaces can again be constructed.  So far, so good.  But this is only half the truth. When it comes to practical applications, clearly only discrete data can be handled, and therefore a suitable discretization, e.g., an atomic decomposition and/or a frame is needed. For the Fr\'echet setting, this turns out to be surprisingly difficult. This problem has been intensively studied in \cite{DDDSSTV}, but the results look in a certain sense ugly and suboptimal. Then the question arises what might be the reason for this. One conjecture could be:  the test and the distribution spaces are Fr\'echet spaces, but the coorbit spaces are Banach spaces. Maybe it is more natural to define the coorbits also as Fr\'echet spaces?  In this paper, we follow this line of research.  And indeed, it turns out that under some reasonable conditions the existence
 of an atomic decomposition can be established in quite a natural way.  This supports our feeling that we follow a feasible 
 path.

\paragraph{Outline}
The paper is structured in the following way. 
We start by recalling basic facts from coorbit theory on Fr\'echet spaces in Section   \ref{sec:overview}.  Then, in Section 3, we introduce our new concept of coorbit
spaces as Frechet spaces.  Then, in Section \ref{atomic}, we state and prove our main result which  provides atomic decompositions 
in Fr\'echet coorbit spaces.  In Section \ref{sec:ex}, we provide and analyze  two examples of coorbit Fr\'echet spaces,
namely Shannon wavelet spaces and  generalized modulation spaces. Some useful facts are proved in the appendix. 


\section{An Overview}\label{sec:overview}
 
In this section we introduce the notation  and  the   basics of  coorbit theory on Fr\'echet spaces, as introduced in~\cite{DDDLSTV} and then 
summarized in \cite{DDDSSTV}.   
For details the reader is referred to the aforementioned literature.

Throughout this paper $G$ denotes a fixed locally compact second countable
group with a left  Haar measure  and  with modular function $\Delta$.
We  denote by  $\int_G f(x)~\d x$  the integrals with respect to the Haar measure and  by $ L_0(G) $ the space of Borel-measurable functions.
Given $f\in L_0(G)$ the functions $\check{f}$ and $\overline{f}$ are
$$
  \check{f}(x)=f(x^{-1}), \qquad \overline{f}(x)=\overline{f(x)},
$$
and for all $x\in G$ the left and right regular
representations $\lambda$ and $\rho$ act on $f$ as

\begin{alignat*}{2}
  \lambda(x)f\,(y) & =f(x^{-1}y) \qquad &\text{a.e }y\in G,\\
\rho(x)f\,(y) & =f(yx) \qquad &\text{a.e }y\in G.
\end{alignat*}
The convolution $f\ast g$ between $f,g\in  L_0(G) $ is the function
\[
  f\ast g(x)
  = \int_G f(y) g(y^{-1} x) ~dy
  = \int_G f(y) \cdot (\lambda(x)\check{g})(y) ~dy
  \qquad \text{a.e. } x \in G,
\]
provided that the function
$y \mapsto f(y) \cdot (\lambda(x) \check{g})(y)$ is integrable for almost all $x \in G$.


We fix a continuous weight $w : G \to (0, \infty)$ satisfying
\begin{subequations}
  \begin{align}
    w(xy) & \leq w(x)w(y), \label{eq:ControlWeightSubmultiplicative} \\
    w(x) & = w(x^{-1})    \label{eq:ControlWeightSymmetric}
  \end{align}
for all $x,y \in G$. It is worthwhile observing that the above properties imply that
\begin{equation}
  \inf_{x\in G} w(x) \geq 1 \label{eq:ControlWeightBoundedBelow}.
\end{equation}
\end{subequations}

For all $p\in[1,\infty)$ the corresponding weighted Lebesgue space is  the separable Banach space
\[
  L_{p,w}(G) = \left\{f\in  L_0(G)  ~\middle|~ \int_G \lvert{w(x)f(x)}\rvert^p ~dx<\infty\right\}
\]
with norm
\[
  \lVert f \rVert_{L_{{p,w}}}^p = \int_G \lvert{w(x)f(x)}\rvert^p ~dx,
\]
and the obvious modifications for $L_\infty(G)$, which
  however is not separable. When $w \equiv 1$, that is, in the unweighted case, we simply write $L_p(G)$.

With terminology as in \cite{DDDLSTV} we choose,
as a {\em target space} for the coorbit theory, the space
\[
  \mathcal{T}_w =\bigcap_{1<p<\infty} L_{p,w}(G).
\]
We recall some basic properties of $\mathcal{T}_w$ (see Theorem 4.3 of \cite{DDDLSTV}, which is based on results in \cite{damuwe70}).
We endow $\mathcal{T}_w$ with the
(unique) topology such that a sequence $(f_n)_{\in\mathbb{N}}$ in $\mathcal{T}_w$
converges to $0$ if and only if $\lim_{n\to+\infty} \lVert f_n \rVert_{L_{p,w}}=0$
for all $1<p<\infty$. With this topology, $\mathcal{T}_w$ becomes a reflexive
Fr\'echet space. The (an
ti)-linear dual space of
$\mathcal{T}_w$ can be identified with
\[
  \mathcal{U}_w = {\operatorname{span}}\bigcup_{1<q<\infty} L_{q,{w^{-1}}}(G)
\]
under the pairing
\begin{equation}
  \int_G \Phi(x)\overline{f(x)}~dx=\langle{\Phi},{f}\rangle_w,\qquad
  \Phi \in \mathcal{U}_w,\,f\in\mathcal{T}_w.\label{eq:67}
\end{equation}

Take now a (strongly continuous) unitary representation $\pi$ of $G$ acting
on a separable complex Hilbert space $\mathcal{H}$ with scalar product
$\langle{\cdot},{\cdot}\rangle_\mathcal{H}$ linear in the first entry. We assume that $\pi$ is
reproducing, namely that  there exists a vector $u \in \mathcal{H}$ such that the
corresponding voice transform
\[
  Vv(x) = \langle{v},{\pi(x)u}\rangle_\mathcal{H},
  \qquad
  v \in \mathcal{H},\, x \in G,
\]
is an isometry from $\mathcal{H}$ into $L_2(G)$. In this case, $u$ is referred to as an {\it admissible vector}. We observe that this implies that
$V$ is injective, when
$\operatorname{span} \left\{\pi(x)u\right\}_{x \in G}$ is dense in $\mathcal{H}$.
We denote by $K$ the reproducing kernel associated to $u$, namely
\begin{equation}
  K(x)
  =Vu(x)
  =\langle{u},{\pi(x)u}\rangle_\mathcal{H},
  \qquad
  x \in G, v \in \mathcal{H}.
\label{eq:KernelDefinition}
\end{equation}
It is a bounded continuous function and enjoys the fundamental properties:
\begin{subequations}
  \begin{align}
    & \overline{K} = \check{K}, \label{eq:32a} \\
    & \sum_{i,j = 1}^n c_i \overline{c_j} K(x_i^{-1} x_j) \geq 0,
      \qquad
      c_1, \ldots, c_n \in \mathbb{C},\
      x_1,\ldots,x_n\in G, \label{eq:32b} \\
    & K \ast K = K \in L_2(G). \label{eq:32c}
  \end{align}
\end{subequations}

For the remainder of the paper, we will work under the following
basic hypothesis.
\begin{assumption}\label{assume:KernelAlmostIntegrable}
  We assume $K \in \mathcal{T}_w$, {\em i.e},
  \begin{equation}\label{eq:KernelAlmostIntegrable}
    K \in L_{p,w}(G)
    \text{ for all } 1 < p < \infty.
  \end{equation}
\end{assumption}

Observe that if $w^{-1}$ belongs to $L_q(G)$ for some $1 < q < \infty$, then
        H\"older's inequality shows $K \in L_1(G)$, but in general
        $K \notin L_{1,w}(G)$.
       Indeed, in many interesting examples $w$ is independent of one or more
        variables, so that $w^{-1} \not\in L_q(G)$ for all $1 < q < \infty$. Thus, in these circumstances there is no guarantee that $K \in L_1(G)$ and in fact in some known instances (see below) this does not happen.

%

Another key ingredient of the coorbit theory is given by the \emph{test space} $\mathcal{S}_w$, defined  as
\begin{equation}
  \mathcal{S}_w = \left\{
                v\in\mathcal{H}
                ~\middle|~
                Vv \in L_{p,w}(G) \text{ for all } 1 < p < \infty
              \right\},
  \label{eq:43}
\end{equation}
which becomes a locally convex topological vector space with the
family of semi-norms
\begin{equation}
  \lVert{v}\rVert_{p,\mathcal{S}_w} = \lVert{Vv}\rVert_{L_{p,w}}.
  \label{eq:29}
\end{equation}
We recall the main properties of $\mathcal{S}_w$.

\begin{theorem}[Theorem 4.4 of \cite{DDDLSTV}]\label{intersections}
Under the assumption~\eqref{eq:KernelAlmostIntegrable}, the following properties hold true.
  \begin{enumerate}
    \item{} The space $\mathcal{S}_w$ is a reflexive Fr\'echet space,
          continuously and densely embedded in $\mathcal{H}$.

    \item{}\label{2.1b} The representation $\pi$ leaves $\mathcal{S}_w$ invariant and its
         restriction to $\mathcal{S}_w$ is a continuous representation.

    \item{} The space $\mathcal{H}$ is continuously and densely embedded into the 
          (anti)-linear dual $\mathcal{S}'_w$, where both spaces are endowed with the
          weak topology.

    \item{} The restriction of the voice transform $V:\mathcal{S}_w\to \mathcal{T}_w$
          is a topological  isomorphism from $\mathcal{S}_w$ onto the
          closed subspace $\mathcal{M}^{\mathcal{T}_w}$ of $\mathcal{T}_w$, given by
          \[
            \mathcal M^{\mathcal{T}_w} = \left\{f\in \mathcal{T}_w ~\middle|~ f \ast K = f\right\},
          \]
          and it intertwines $\pi$ and $\lambda$.

    \item\label{bar}{} For every $f\in \mathcal{T}_w$, there exists a unique element
          $\pi(f)u\in\mathcal{S}_w$ such that
          \[
            \langle{\pi(f)u},{v}\rangle_\mathcal{H}
            = \int_G f(x) \langle{\pi(x)u},{v}\rangle_\mathcal{H}~dx
            = \int_G f(x) \overline{Vv(x)}~dx,
            \qquad v\in\mathcal{H}.
          \]
          Furthermore, for every $f\in \mathcal{T}_w$ 
          \[
            V\pi(f)u = f \ast K,
          \]
           the map
          \[
            \mathcal{T}_w \ni f \mapsto \pi(f)u \in \mathcal{S}_w
          \]
          is continuous and its restriction to $\mathcal M^{\mathcal{T}_w}$
          is the inverse of $V$.
 \end{enumerate}
\end{theorem}

Here and in what follows the notation $\pi(f)u$ is motivated by the fact that any function $f\in L_1(G)$ defines a bounded operator $\pi(f)$ on $\mathcal{H}$, which is weakly given by
  \[
    \langle{\pi(f) v},{v'}\rangle_\mathcal{H} = \int_G f(x) \langle{\pi(x)v},{v'}\rangle_\mathcal{H} ~dx,
    \qquad v,v'\in\mathcal{H},
  \]
  see for example Sect. 3.2 of \cite{fol95}. However, if $f\not\in L_1(G)$,
  then in general $\pi(f)v$ is well defined only if $v=u$, where
  $u$ is an admissible vector for the representation $\pi$.


Recalling that the (anti-)dual of $\mathcal{T}_w$ is $\mathcal{U}_w$ under the
  pairing~\eqref{eq:67}, we denote by $\,^t\!{V}$ the contragradient  map
  $\,^t\!{V}: \mathcal{U}_w\to \mathcal{S}'_w$ given by
\[
\langle{\,^t\!{V}\Phi },{v}\rangle_{\mathcal{S}_w}= \langle{\Phi},{Vv}\rangle_{w},\qquad
\Phi\in\mathcal{U}_w,\, v\in\mathcal{S}_w.
\]

As usual, we extend the voice transform from $\mathcal{H}$ to the (anti-)dual $\mathcal{S}'_w$
of $\mathcal{S}_w$, where $\mathcal{S}'_w$ plays the role of the space of distributions.
For all $T \in \mathcal{S}_w'$ we define the {\it extended voice transform} of $T$ by
\begin{equation}
  V_e T(x) = \langle{T},{\pi(x)u}\rangle_{\mathcal{S}_w},
  \qquad x\in G,
\label{EVT2}
\end{equation}
which is a continuous function on $G$ by item~\ref{2.1b}) of the previous
theorem and $\langle{\cdot},{\cdot}\rangle_{\mathcal{S}_w}$ denotes the pairing between $\mathcal{S}_w$.
Here $\mathcal{S}'_w$, whereas $\langle{\cdot},{\cdot}\rangle_w$ is the pairing between $\mathcal{T}_w$ and
$\mathcal{U}_w$.

We summarize the main properties of the extended voice transform in
the next theorem.

\begin{theorem}[Theorem 4.4 of \cite{DDDLSTV}]\label{intersections-1}
  Under assumption~\eqref{eq:KernelAlmostIntegrable}, the following facts hold true.
  \begin{enumerate}
    \item{} For every $\Phi \in \mathcal{U}_w$ there exists a unique element
          $\pi(\Phi)u\in\mathcal{S}_w'$ such that
          \[
            \langle{\pi(\Phi)u},{v}\rangle_{\mathcal{S}_w}
            = \int_G \Phi(x) \langle{\pi(x)u},{v}\rangle_{\mathcal{H}}\,dx
            = \int_G \Phi(x) \overline{Vv(x)}\,dx,
            \qquad v \in \mathcal{S}_w.
          \]
          Furthermore, it holds that
          \[
            V_e \pi(\Phi)u = \Phi \ast K.
          \]

  \item{}\label{2.2b} For all $T\in\mathcal{S}_w'$ the extended voice transform  $V_eT$ is in
        $\mathcal{U}_w$ and satisfies
        \begin{align}
          & V_e T = V_e T \ast K, \label{eq:ExtendedVoiceTransformReproducing} \\
          & \langle{T},{v}\rangle_{\mathcal{S}_w} = \langle{V_e T},{V v}\rangle_w,
          \qquad v \in \mathcal{S}_w.
          \label{eq:ExtendedVoiceTransformDuality}
        \end{align}

  \item{} The extended voice transform $V_e$ is injective, it is continuous
        from $\mathcal{S}_w'$ into $\mathcal{U}_w$ (when both spaces are endowed with the
        strong topology), its range is the closed subspace
        \begin{equation}
          \mathcal M^{\mathcal{U}_w}
          = \left\{\Phi \in \mathcal{U}_w ~\middle|~ \Phi \ast K = \Phi\right\}
          = \operatorname{span}
              \bigcup_{p \in (1,\infty)}
                \mathcal{M}^{L_{p,w}(G)}
          \subset L_{\infty,{w^{-1}}}(G)
          \label{eq:38}
        \end{equation}
      and it intertwines the contragradient representation of $\pi_{|\mathcal{S}_w}$ and
      $\lambda_{|\mathcal{U}_w}$.

  \item\label{sotto}{} The map
        \[
          \mathcal M^{\mathcal{U}_w} \ni \Phi \mapsto \pi(\Phi) u \in \mathcal{S}_w'
        \]
        is the left inverse of $V_e$ and coincides with the restriction of the map
        $\,^t\!{\,V}$ to $\mathcal M^{\mathcal{U}_w}$, namely
        \begin{equation}
          V_e( \,^t\!{\,V}\Phi)
          = V_e\pi(\Phi)u
          = \Phi,
          \qquad \Phi \in \mathcal M^{\mathcal{U}_w}.
          \label{eq:39}
        \end{equation}

  \item{} Concerning the double inclusion $\mathcal{S}_w\hookrightarrow\mathcal{H}\hookrightarrow\mathcal{S}'_w$, we have
        \[
          \mathcal{S}_w
          = \left\{T \in \mathcal{S}_w' ~\middle|~ V_e T \in \mathcal{T}_w\right\}
          = \left\{\pi(f)u ~\middle|~ f\in \mathcal{M}^{\mathcal{T}_w} \right\}.
        \]
  \end{enumerate}
\end{theorem}

Item~\ref{2.2b}) of the previous theorem states that
the voice transform of any distribution $T\in\mathcal{S}_w'$ satisfies the
reproducing formula~\eqref{eq:ExtendedVoiceTransformReproducing} and uniquely
defines the distribution $T$ by means of the reconstruction
formula~\eqref{eq:ExtendedVoiceTransformDuality}, i.e.
\[
  T = \int_G \langle{T},{\pi(x)u}\rangle_{\mathcal{S}_w} \pi(x)u~dx,
\]
where the integral is a Dunford-Pettis integral with respect to the duality
between $\mathcal{S}_w$ and $\mathcal{S}_w'$, see, for example,
Appendix~3 of \cite{fol95}.

\section{Fr\'echet Coorbit Spaces}

In this section,  we present our new construction of coorbit spaces as Fr\'echet spaces.  For simplicity, we will  restrict ourselves 
to the case where  the  coorbit spaces are related with intersections of weighted $L_p$-spaces.  In particular, we will show
that the most important result  in the realm of coorbit space theory, the Correspondence Principle, 
 also holds in our case, see Theorem \ref{Ycoorbit} below. 
 
We  fix  a $w$-moderate weight $m$, i.e. a continuous function $m : G \to (0, \infty)$ such that
\begin{align}
  &m(xy) \leq w(x) \cdot m(y) \label{eq:ctrl1}\\
&m(xy) \leq m(x) \cdot w(y)\label{eq:ctrl2}.
\end{align}
%
The above conditions are equivalent to
\begin{align*}
m(xyz) \leq w(x)\cdot m(y)\cdot w(z) \quad \text{for all}~x,y,z\in G
\end{align*}
up to the constant $w(e)$. It is worth observing that if $m$ is a $w$-moderate weight on $G$, then so is $m^{-1}$, see Lemma~4.1 in~\cite{DDDSSTV}.
%
%

Appealing again to the terminology used in \cite{DDDLSTV}, we
choose as a {\em model space} $Y$ for the coorbit space theory
the Fr\'echet space (see Theorem~\ref{Yprop} below)
\[
Y =\bigcap_{1<p<\infty} L_{p,m}(G).
\]
instead of  the single Banach space $L_{r,m}(G)$ as it is done in \cite{DDDSSTV}. The space $Y$ is endowed with the initial topology that makes all the inclusions $\iota_p\colon Y\hookrightarrow L_{p,m}(G)$ continuous. It is worth pointing out that the notation ${\mathcal T}_m$ to denote $Y$ would be improper because $m$ is neither necessarily submultiplicative nor symmetric. In particular, the general properties enjoyed by ${\mathcal T}_w$  are not automatically satisfied by $Y$,  but hold in the sense clarified in the next result. 
\begin{theorem}\label{Yprop}
  \begin{enumerate}
    \item{} The space $Y$ is a Fr\'echet space  continuously embedded into $L_0(G)$ via the natural inclusion 
    $j\colon Y\hookrightarrow L_0(G)$. Furthermore
 \[
    Y'=Y^\sharp=\bigcup_{1<p<\infty} L_{p,m^{-1}}(G),
\]
where  $Y^\sharp=\left\{g\in L_0(G)~\middle| gj(f)\in L^1(G) \text{ for all }f\in Y\right\}$ is the K\"othe dual of $Y$.
  \item{} The left regular representation leaves $Y$ invariant and its restriction to $Y$  is a continuous representation.
 \item{} For all $f\in\mathcal{T}_w$ and $F\in Y$, it holds that $fF\in L_1(G)$.
  \end{enumerate}
\end{theorem}
\begin{proof} {\it 1.} This is a consequence of Theorem~4.3 in~\cite{DDDLSTV} with $w$ replaced by $m$ because the proof does not depend on the properties  \eqref{eq:ControlWeightSubmultiplicative} and \eqref{eq:ControlWeightSymmetric} of $w$.

{\it 2.} This is a consequence of Theorem~4.3 together with Lemma~4.2  in~\cite{DDDLSTV}. The proof of the latter is based on the sub-multiplicative property \eqref{eq:ControlWeightSubmultiplicative} of $w$, which may be replaced by \eqref{eq:ctrl1}, namely
\[
\int_G|m(y)f(x^{-1}y)|^p\,\d y=\int_G|m(xy)f(y)|^p\,\d y\leq w(x)^p\int_G|m(y)f(y)|^p\,\d y.
\]

{\it 3.} Theorem~4.3 in~\cite{DDDLSTV} shows that $Y\subseteq{\mathcal U}_w={\mathcal T}_w^\sharp$. Further, the symmetry property \eqref{eq:ControlWeightSymmetric} gives
\[
m(e)=m(xx^{-1})\leq m(x)w(x^{-1})=m(x)w(x)
\]
and hence
\[
w^{-1}(x)\leq\frac{m(x)}{m(e)}.
\]
The definition of K\"othe dual then entails  that  
$fF\in L_1(G)$ for every $f\in {\mathcal T}_w$ and every $F\in Y$.
\end{proof}

Following the usual steps of the theory, the coorbit space corresponding to $Y$ is defined by
\begin{equation}
 \operatorname{Co}(Y) = \left\{T\in \mathcal{S}'_w ~\middle|~ V_e T \in Y \right\},
 \label{eq:21}
\end{equation}
i.e.,   $\operatorname{Co}(Y)$ is the  { vector} space
$$  \operatorname{Co}(Y) = \left\{T\in \mathcal{S}'_w ~\middle|~ V_e T \in  \bigcap_{ 1<p<\infty}L_{p,m}(G) \right\}=\bigcap_{ 1<p<\infty}
 \operatorname{Co}(L_{p,m}(G)).$$

We summarize the main properties of $\operatorname{Co}(Y)$ in the following result, which is an adaptation of Proposition~4.1 of~\cite{DDDSSTV} to the present setup.
\begin{theorem}\label{Ycoorbit}
  \begin{enumerate}
    \item{} The image under the extended voice transform of the coorbit space is
    \[
    V_e({\operatorname{Co}(Y)} )=\left\{F\in Y~\middle|~ F*K=K\right\},
    \]
 a closed subspace of $Y$ denoted  ${\mathcal M}^Y $. 
  \item{} For all $f\in{\mathcal M}^Y$, the Fourier  transform $\pi(f)u$  given  as in Theorem~\eqref{intersections-1} by
  \[
  \langle\pi(f)u,v\rangle_{{\mathcal S}_w}=\int_G f(x) \langle\pi(x)(u),v\rangle_{\mathcal H}\,\d x,\qquad v\in{\mathcal S}_w,
  \]
  belongs to the coorbit space $\operatorname{Co}(Y)$.
 \item{} The extended voice transform on $\operatorname{Co}(Y)$ and the Fourier transform  on ${\mathcal M}^Y$ at $u$ satisfy 
 \begin{align}
 V_e\left(\pi(f)u\right)&=f,\qquad f\in{\mathcal M}^Y\\
 \pi(V_e T)u&=T,\qquad T\in\operatorname{Co}(Y).
 \end{align}
   \item{} $V_e$ is a bijection from $\operatorname{Co}(Y)$  onto ${\mathcal M}^Y$ whose inverse is the map $f\mapsto\pi(f)u$.
  \end{enumerate}
\end{theorem}
\begin{proof} Proposition~2.2 in~\cite{DDDLSTV} ensures that for every $f\in Y$ the convolution $f*K$ is well defined. Indeed, the asumption  that $K\in Y^\sharp$ used in the proof is equivalent to assuming that $fK\in L_1(G)$ for every $f\in Y$, and this is satisfied by the first item in Theorem~\ref{Yprop} because $K\in{\mathcal T}_w$. Further, Proposition~2.4 in~\cite{DDDLSTV}, whereby one has to take $E={\mathcal S}_w$, ensures that $\pi(f)u$ is a well defined  element of ${\mathcal S}_w'$ for every $f\in Y$. Indeed,  the requirement (13) of that proposition is satisfied because it consists in the fact that $fVv\in L_1(G)$ for all $f\in Y$ and every $v\in {\mathcal S}_w$, which is implied by the last item in Theorem~\ref{Yprop} because $Vv\in{\mathcal T}_w$. The content of the present theorem is then the same as that of Proposition~2.6 in~\cite{DDDLSTV}, where again one has to put  $E={\mathcal S}_w$ and also $F=Y$. Indeed, as already observed, condition (13) in~\cite{DDDLSTV} is satisfied. Hence (15) in~\cite{DDDLSTV} holds true for every $f\in Y$, namely $V_e\left(\pi(f)u\right)=f*K$, but since ${\mathcal M}^Y\subseteq Y$ we have 
\[
V_e\left(\pi(f)u\right)=f*K=f
\]
for every $f\in{\mathcal M}^Y$. Therefore $V_e\left(\pi(f)u\right)\in Y$ and hence $\pi(f)u\in\operatorname{Co}(Y)$ for every $f\in{\mathcal M}^Y$. This proves 2 and the first equality in 3.
 In order to conclude it is enough to show that  $V_eT*K=V_eT$ for every  $T\in\operatorname{Co}(Y)$. If so, then Lemma~2.5 in~\cite{DDDLSTV} tells us that this is equivalent to the second equality in~3. The fact that $V_eT*K=V_eT$ for every  $T\in\operatorname{Co}(Y)$  is exactly~(21a) in Proposition~2.6 in~\cite{DDDLSTV}.  
\end{proof}

{ By means of the above theorem, which will be referred to as the Correspondence Principle, it is possible to endow 
$\operatorname{Co}(Y)$ with a natural Fr\'echet topology, simply by transferring via (the inverse of) the extended voice transform $V_e$ the  Fr\'echet topology that ${\mathcal M}^Y$ inherits as closed subspace of~$Y$. The idea of transferring the topology using the extended voice transform mimics the strategy that is taken in the Banach space setup where the coorbit norm is defined using the Correspondence Principle.}

\section{Atomic Decomposition} \label{atomic}
After the { topology and the properties of the} generalized coorbit spaces { $\operatorname{Co}(Y)$ have been} established, the next task clearly is the construction of suitable discretizations of these spaces, i.e., to provide  some kind of an atomic decomposition. The first question is what an atomic decomposition of  a Fr\'echet space should be.  In the classical coorbit setting, this means that there exists a countable family of functions in the coorbit space, the atoms, such that every element can be written as a linear combination of these atoms, and that norm equivalences of the coorbit norms and weighted sequence norms of the coefficients  hold. Now, these norm equivalences  can be interpreted as the fact that  linear mappings from  the coorbit spaces to  the weighted sequence spaces  are continuous, and vice versa.  On a   Fr\'echet space,  one does not have a norm, but nevertheless a topology so that continuity of mappings is well-defined.  Therefore, the following definition arises naturally.

\begin{definition}   Let ${\mathcal F}$  be a Fr\'echet space, and let $b(\Lambda)$ be  a sequence space. A set  $\{a_{\lambda}\}_{\lambda \in \Lambda}  \subset {\mathcal F}$ gives rise to an {\rm atomic decomposition}
if every $f \in {\mathcal F}$  has an expansion
$$ f=\sum_{\lambda  \in \Lambda}  c_{\lambda}(f)a_{\lambda}$$
and the mappings
\begin{itemize}
 \item[(i)] Analysis:
$$  A~:~~{\mathcal F}  \longrightarrow b(\Lambda), \qquad A(f)=\{c_{\lambda}(f)\}_{\lambda   \in \Lambda} $$
\item[(ii)]  Synthesis: 
$$  S~:~~ b(\Lambda)  \longrightarrow{\mathcal F}, \qquad S(\{d_{\lambda}\}_{\lambda \in \Lambda})=\sum_{\lambda \in \Lambda} d_{\lambda} a_{\lambda}$$\vspace{-2mm}
\end{itemize}
are well-defined and continuous.
\end{definition}

{ To establish such an atomic decomposition for
  $\operatorname{Co}(Y)$, observe that by Theorem~\ref{Ycoorbit} {,  and thus by the topology defined on $\operatorname{Co}(Y)$},  it  is enough to prove an atomic decomposition for $\mathcal M^Y$. As in
the classical setting, we are looking at atoms of the form {  $L_{g_i}K$}
where $\{g_i\}_{i\in I}$ is a suitable family of elements in $G$ and
the corresponding coefficients are given by
\[
a_i(f) =\langle f,\phi_i\rangle,
\]
where $\{\phi_i\}_{i\in }$ is a suitable partition of unity of $G$.}
To this end, we start with a $Q$-dense set in the group $G$, whereby
$Q$ is a compact set  with nonvoid interior and  $e \in Q$. Then, a countable family $X= (g_i)_{i\in \II}$  is 
said to be {\em $Q$-dense} if $\cup_{i\in \II} g_i Q=G$.  Moreover, let $\Phi =(\phi_i, X, Q)$ denote a partition of unity subordinate to the $Q$-dense set $X$, i.e., 
\begin{equation}  \label{partition}
\text{supp}\phi_i \subset g_i Q, \quad 0\leq \phi_i \leq 1~\text{for all}~ i\in \II, \quad \sum _{i \in \II} \phi_i=1.
\end{equation}
Based on such a partition of unity, our discretization operator is defined as follows: 
\begin{equation} \label{defJphi}
J_ \phi (F) := \sum_{i\in\II}\langle F, \phi_i \rangle L_{g_i} K
\end{equation}
The first step is to show the continuity of   this operator.  We need the so-called  {\em $Q$-oscillation}, { which is defined by:}
\begin{equation} \label{defosc}
{ \operatorname{osc}_Q(F)} := \sup_{q \in Q}  |F(qg)-F(g)| .
\end{equation}

\begin{lemma}  \label{continuity} Let { $\|\operatorname{osc}_Q(K) \|_{L_{q,w}}< \infty$} for all  { $1<q<\infty$}. 
Then the  map { $J_{\phi}~:~{\mathcal M}^Y \rightarrow {\mathcal M}^Y$} defined by
\[
 J_ \phi (F) := \sum_{i\in\II}\langle F, \phi_i \rangle L_{g_i} K
\]
is continuous. 
\end{lemma}
\begin{proof}
 Let { $F \in {\mathcal M}^Y$}. Then  $F \in L_{ p,m}$ for all $p>1$. { If we} 
 show that $Id-J_{\phi}$ is continuous, then clearly also $J_{\phi}$ is continuous. We start with a pointwise estimate. 
 The reproducing kernel property  and the fact that $\Phi$ is a partition of unity imply
  \begin{eqnarray}
 |(Id-J_{\phi})(F)(g)|  
 &=&|\int _G F(h)K(h^{-1}g) dh -\sum_{i\in\II} \int_G  F(h)\phi_i(h) K(g_i^{-1} g) dh |\\
  &=&|\int_G F(h)(\sum_{i\in\II} \phi_i(h)(K(h^ {-1}g)-K(g_i^{-1} g) dh)|\\
  & \leq &\int_G (\sum_{i\in\II} |F(h)|\phi_i(h)\sup_{q \in Q}|(K(h^ {-1}g)-K(q h^{-1} g) |dh)\\
 & \leq &\int_G |F(h)|\sum_{i\in\II}\phi_i(h){\operatorname{osc}_Q(K)}(h^ {-1}g) dh \\ 
 &\lesssim & (|F| \ast {\operatorname{osc}_Q(K)})(g) .
 \end{eqnarray}
 
 Then  by the generalized weighted Young inequality  \eqref{generalizedYoung} we obtain, for { $1/p+1/q=1+1/r$,}
\[
\| F- J_{\phi}(F)\|_{L_{r,m}}\lesssim 
 \| |F| \ast\operatorname{osc}_Q(K)\|_{L_{r,m}}\lesssim \|F\|_{L_{p,m}} \| \operatorname{osc}_Q(K)\|_{L_{q,w}} . 
 \]
 By our assumptions on { $\operatorname{osc}_Q(K)$} and the definition of { Fr\'echet} topology the result follows. 
 \end{proof}

Now, we are in position to state and to prove the main result of this paper. 

\begin{theorem} \label{mainresult}
Assume that  $J_{\Phi}$ is injective with continuous left inverse  $J_{\Phi}^{-1}$.  Then
\begin{itemize}
\item[(i)]Every { $T \in\operatorname{Co}(Y)$} can be represented as
\begin{equation}  \label{hurra}
 T= \sum_{{i\in\II}}\underbrace{\langle V_{\psi}(T), \phi_{i}\rangle}_{\text{coefficients}}\underbrace{V_{\psi}^{-1} J_{\Phi}^{-1} L_{g_{i}}K}_{{\text{atoms}}}.
 \end{equation}
 \item[(ii)] Set $$b(\Lambda)=\bigcap_{{ 1<p<\infty}}\ell_{p,m}.$$ 
If in addition    $\|{\operatorname{osc}_Q}(J_{\Phi}^{-1}K)\|_{L_{p,w}} <\infty$ for all {$1<p<\infty$},  then 
$$\{ V_{\psi}^{-1} J_{\Phi}^{-1} L_{g_{i}}K\}_{i \in \II}$$
gives rise to an atomic decomposition { of $\operatorname{Co}(Y)$}.
\end{itemize}
\end{theorem}

\begin{proof}We start by showing (i). 
\begin{eqnarray}
F &=& J_\phi^{-1} J_\phi F\\
&=& J_\phi^{-1}\left(\sum_{i\in\II}\langle F, \phi_i \rangle L_{g_i} K \right)\\
&=& \sum_{i\in\II}\langle F, \phi_i \rangle J_\phi^{-1} L_{g_i} K ~,
\end{eqnarray}
where $F\in\cM_m$,  the reproducing kernel space. Then for { $T\in\operatorname{Co}(Y)$} we get by the Correspondence Principle
i .e., by applying  Theorem \ref{Ycoorbit},
 \begin{eqnarray}
 V_\psi(T) &=& J_\phi^{-1} J_\phi V_\psi(T)\\
&=& J_\phi^{-1}\left(\sum_{i\in\II}\langle V_\psi(T), \phi_i \rangle L_{g_i} K \right)\\
&=& \sum_{i\in\II}\langle V_\psi(T), \phi_i \rangle J_\phi^{-1} L_{g_i} K 
\end{eqnarray}
and therefore,
\begin{eqnarray}
T &=&  V_\psi^{-1} V_\psi(T) \\
&=& \sum_{i\in\II}\underbrace{\langle V_\psi(T), \phi_i \rangle}_{\text{coefficients}}  \underbrace{V_\psi^{-1} J_\phi^{-1} L_{g_i} K }_{\text{atoms}}
\end{eqnarray}
and (i) is shown. 

The next step is to prove (ii).
We have to  show that:
\begin{eqnarray}
S:~ b(\Lambda) &\longrightarrow& { \cM^Y}\\
\{c_i\}_{i\in\II} &\mapsto&  \sum_{i\in\II} c_i J_\phi^{-1} L_{g_i} K \\
&& \text{is well-defined and continuous}\\
A:~{\operatorname{Co}(Y)} &\longrightarrow& b(\Lambda)\\
T &\mapsto&  \{\langle V_\psi T,\phi_i\rangle\}_{i\in \II}\\
&& \text{is continuous.~~}
\end{eqnarray}
{ If this is established, then} the result follows by the Correspondence Principle. We start with 
 the operator $A$.  By  \cite{DDGL}   p. 98, eq. (3.27), we observe  that 
$$\|\{\langle F,\phi_i \rangle\}_{i\in \II}\|_{\ell_{p,m} }\lesssim\|F\|_{L_{p,m}}.$$
Then,
$$\|\{\langle V_\psi(T),\phi_i \rangle\}_{i\in \II}\|_{\ell_{p,m}} \lesssim\|V_\psi(T)\|_{L_{p,m}}\lesssim\|T\|_{\cH_{p,m}}~.$$
Now we consider the operator $S$. At first, we observe
\begin{eqnarray}
| \sum_{i\in\II} c_i (J_\phi^{-1}  K)(g_i^{-1}\circ l) | \lesssim
\left(\sum_{i\in\II}|c_i| \chi_{g_i Q} \right)*\left({\rm osc}_Q (J_\phi^{-1}K)+|J_\phi^{-1}K|\right)(l)
\end{eqnarray}
see \cite{DDGL},  p. 100, Lemma 3.18 for details. Then,
\begin{eqnarray}
\| \sum_{i\in\II} c_i (J_\phi^{-1}  K)(g_i^{-1}\circ \cdot)\|_{L_{p,m}}\lesssim
\left\|\left(\sum_{i\in\II}|c_i| \chi_{g_i Q} \right)*\left({\rm osc}_Q (J_\phi^{-1}K)+|J_\phi^{-1}K|\right)(\cdot)   \right\|_{L_{p,m}}~,
\end{eqnarray}
and applying the  generalized weighted Young inequality with $1/p + 1 = 1/q + 1/{q'}$ yields
\begin{eqnarray}
&&\hspace*{-3cm}\left\|\left(\sum_{i\in\II}|c_i| \chi_{g_i Q} \right)*\left({\rm osc}_Q (J_\phi^{-1}K)+|J_\phi^{-1}K|\right)(\cdot)   \right\|_{L_{p,m}}\\
&\lesssim& \left\|\sum_{i\in\II}|c_i| \chi_{g_i Q}(\cdot) \right\|_{L_{q,m}}\cdot
\left\|{\rm osc}_Q (J_\phi^{-1}K)+|J_\phi^{-1}K|(\cdot)   \right\|_{L_{q',w}}~.
\end{eqnarray}
By our assumptions, $\|{\rm osc}_Q (J_\phi^{-1}K)\|_{L_{ q',w}}$ and $\|(J_\phi^{-1}K)(\cdot)\|_{L_{q',w}}$  are finite. Therefore, 
by using Lemma 3.18 in  \cite{DDGL}, we obtain
$$ \| \sum_{i\in\II} c_i (J_\phi^{-1}  K)(g_i^{-1}\circ \cdot)\|_{L_{p,m}}\lesssim \|  (c_i)_{i \in I}  \| _{\ell_{q,m}} $$
and (ii) is shown. \end{proof}
%
\begin{remark} \label{ discussion1}
\begin{itemize}
\item[(i)] It is one of the advantages of our approach that the  norms of  the $Q$-oscillations  corresponding to $K$  
and $J_{\phi}^{-1}K$, respectively, do not have to be uniformly bounded, they just have to be finite.
\item[(ii)]  In many cases, an explicit formula for  $J_{\phi}^{-1}$ will not be available, so that, at first glance, it seems to be
 diifficult to estimate the norms of the $Q$-oscillation of  $J_{\phi}^{-1}K$. However,  although the coorbit spaces might consist of  ugly distributions, the reproducing kernel spaces 
 are usually spaces of nice, smooth functions, so that in many cases the norms of the $Q$-oscillations
  turn out to be finite for all functions in the reproducing kernel space which is clearly sufficient.  
  We also observe this fact  in the examples discussed in Section \ref{sec:ex}.
  \item[(iiii)]  It is  clearly a drawback that currently no general conditions that guarantee the injectivity
   of $J_\phi$ are available. Fortunately, usually this property can be checked directly, see Section \ref{sec:ex}.
   In the classical setting ot coorbit Banach spaces,  injectivity and 
    even surjectivity    is proved by a Neumann series argument.  Of course, the 
    Neumann series setting can be generalized to Fr\'echet spaces, but in our case this series would not converge
    for this would require a  uniform bound of the $Q$-oscillations. 
     \item[(iv)]    From the examples studied in the following section, we observe that, to show the injectivity of  $J_\phi$, it is 
     not absolutely necessary  to choose a very small neighbourhood  $Q$. Nevertheless, in practical 
     applications, a small set $Q$ (which implies a denser set $X= (g_i)_{i\in \II}$)  might be advantageous 
     since  the constants depend, among other  things, on $ \bigl\|{\rm osc}_Q (J_\phi^{-1}K) \bigr\|_{L_{q',w}} $  which most likely
      will grow as $Q$ gets larger.   Moreover, as we will see in  Section \ref{sec:ex} , to show the continuity  of the left inverse,
      sometimes a denser sampling is helpful.    \end{itemize} 
     \end{remark}     
\section{Examples of Coorbit Fr\'echet Spaces}\label{sec:ex}

\subsection{The Shannon Case}\label{sec:shannon}
As a first example, we study spaces of band-limited functions. Let $G$  denote the additive group $\mathbb R$ 
with Lebesgue measure. Then $G$  acts  on the Paley--Wiener space
$${\mathcal H}=B_{\Omega}^2=\{ f \in L_2(\mathbb R) :\mbox{supp}(\hat{f})\subseteq \Omega\}, ~\Omega~ \mbox{compact interval}$$
by translations,  $\pi(b)v(x)=v(x-b)$ The following facts are well-known: 
\begin{theorem}\label{littlehelp} It  holds
  \begin{itemize}
   \item[i)] $\psi \in B^2_\Omega $ is admissible if and only if $|{\widehat{\psi}}| = 1 $ almost everywhere on $\Omega$. Then, the
   reproducing kernel is given by
\[
 K = \langle{\psi},{\pi(\cdot)\psi}\rangle_{{\mathcal H}}={\mathcal F}^{-1}\chi_\Omega .
 \]
  
   \item[ii)]
    If $ \Omega = [-\omega,\omega] $  
    $$ K(b) = 2\omega\,\sinc (2\omega b) , $$
if we use the standard notation $\sinc(x) = \sin(\pi x)/(\pi x)$.
  \end{itemize}
 \end{theorem}
Consequently, the underlying representation is square-integrable,  but of course it is not integrable.  Therefore, it fits into our setting. 
\subsubsection{Fr\'echet coorbit spaces}
The  Fr\'echet coorbit spaces are given by { choosing $Y=\bigcap_{p>1}L_{p,m} (G)$, namely}:
\[
{ \operatorname{Co}(Y)} = \{f\in {\mathcal S}' \mid V_{\psi} f\in  \bigcap_{p>1}L_{p,m}(G)  \} 
\]
and the resulting reproducing kernel spaces  are
 \[
 {\mathcal M^Y}:=\{F\in \bigcap_{p>1}L_{p,m} (G)~|~F \ast K=F\}.
 \]
{
 \begin{remark}
Since the voice transform $V:\mathcal H \to L_2(\R)$ is the canonical
inclusion, $\pi$ is the restriction of $L$ to $\mathcal H$ and
${ \operatorname{Co}(Y)=\mathcal M^Y}$, compare with Proposition 4.8 of
\cite{DDDLSTV}. { For simplicity, for any $\phi\in L_2(\R)$ and any  $\ell\in\Z$ we write $\pi(\ell)\phi=L_\ell\phi=\phi_\ell$.}
 \end{remark}
As a consequence of the results of the previous section we
  have the following theorem, whose proof is the content of the next
  sections.
  \begin{theorem}
 Fix the weights $m=w=1$. If $\psi\in L_2(\mathbb R)$ is such that
 $|\widehat{\psi(x)}|=1$ almost everywhere on $\Omega$, then  $\{
 \pi_\ell\psi\}_{\ell\in\mathbb Z}$ is an atomic decomposition of
 $\mathcal H_m$ and the coefficients are  given by
\[
a_\ell(f) = \langle V_e f,\phi_\ell\rangle  \qquad \ell\in\mathbb Z
\]
where $\phi=(\chi_{[0,1)} \ast \chi_{[0,1)})(\cdot +1)$ is the
centralized cardinal  B-spline $N_2$ and, for any  $\ell\in \mathbb Z$,
$\phi_\ell=L_\ell\phi$.
  \end{theorem}}

{
\begin{remark} \label{ discussion3}
\begin{itemize}
   \item[(i)]The case of weighted spaces can also be treated, at least for the case $w=m$. 
 (Although this might be of limited use since, due to the bad decay properties of the Shannon kernel, only weights of logarithmic type can be handled. )
  Only  the estimation of  
 the norms of the weighted $Q$-oscillation requires some care  { and is achieved using the Sobolev embedding theorem in the way that is explained in the next section. Hence the appropriate inequalities are detailed below.}
   \item[(ii)]    In \cite{DDDLSTV}, it has been shown that in our setting the $L_p$-Banach coorbit spaces coincide 
   with the Paley-Wiener spaces, 
\[
\operatorname{Co}(L_p(\R)) =    B^p_\Omega . 
\]   
 Therefore, we observe the interesting fact that
\[
\operatorname{Co}(Y)= \bigcap_{p>1}B^p_\Omega .
 \]
Consequently, although the discretization of $ \operatorname{Co}(L_p(\R)) =    B^p_\Omega $  turned out to be rather complicated, see
 again \cite{DDDSSTV} for details, an atomic decomposition for their intersection can be 
 constructed in quite a natural way.
 \end{itemize}
 \end{remark} 
 }

{ From now on $w =m=1$.}  We choose $Q=-[1,1]$, so that the $Q$-dense set  is simply given by the integers $\mathbb Z$. The associated partition of unity can be constructed by the integer translates of the centralized cardinal  B-spline 
 $N_2:=(\chi_{[0,1)} \ast \chi_{[0,1)})(\cdot +1)$, i.e, $\phi_k=N_2(\cdot -k), k\in \mathbb Z$. { We observe {\it en passant}  that of course  $\widehat{\phi}(\xi)=\sinc^2(\pi\xi)$.}
\subsubsection{Injectivity}
The first step is to show that 
 \[
 J_\phi = \sum_{i\in\II}\langle F,\phi_i \rangle L_{g_i}K
 \]
 is injective.  Suppose that
 \[
 0=J_\phi(F) = \sum_{k\in \mathbb Z}\langle F,\phi_k\rangle  \frac{\sin (2 \pi (\cdot -k))}{2 \pi (\cdot -k)}\quad \text{for some}\quad F \in {{\mathcal M}^Y}.
 \]
 Applying  the Fourier transform yields
\[
0= \sum_{k \in \mathbb Z}\langle F, \phi_k \rangle{\rm e}^{-2 \pi ik \xi}\chi_{[-1,1)} (\xi)\frac{1}{2}, \quad \text{for almost all}
 \quad \xi \in \mathbb R,
 \]
 and therefore   
\[
\langle  F, \phi_k\rangle =0 \quad \text{for  all} \quad k \in \mathbb Z.
\]
 Then, by using the reproducing kernel property and Plancherel's  theorem  we get
 \begin{eqnarray} 
 0&=& \langle F, \phi_k \rangle =\langle  F \ast K, \phi_k \rangle 
 = \langle \widehat{F\ast K}, \widehat{\phi_k}\rangle \\
 &=& \langle \hat{F} \hat{K}, \hat{\phi} {{\rm e}^{-2\pi i k(\cdot)}} \rangle
 =\int_{-1}^1{\hat{F}(\xi) (\sinc\,\xi)^2{\rm e}^{2\pi i k\xi}{\rm d}\xi}.
  \end{eqnarray}
  This means that all Fourier coefficients  of ${\hat{F}(\xi)(\sinc\,\xi)^2|_{[-1, 1] }}$ vanish,  so that  
  ${\hat{F}(\xi) (\sinc\,\xi)^2|_{[-1, 1] }}=0$. But since  $(\sinc\,\xi)^2$  is  positive and $F$ is band-limited,  $\hat{F}$ is zero, so that finally  $F=0$. Therefore,  $J_{\phi}$  is injective.
  
  
 
 \begin{remark} \label{ discussion2}
 From the above calculations, it it clear that also the case of a finer lattice, say  $h \mathbb Z$, can be handled in a similar fashion. 
 Moreover, other values of  $\omega$ are no problem.
  \end{remark} 
  
\subsubsection{Continuity  of the Left Inverse}

In order to apply Theorem \ref{mainresult}, we have also to verify that the left inverse 
 ${ J}_{\Phi}^{-1}$ is continuous. To this end, we consider a specific partition of unity $\Phi$ and therewith a concrete, sufficiently dense sampling setting. For this specific case we will be able to provide an explicit description of ${J}_{\Phi}$ enabling us to derive a bound for ${ J}_{\Phi}^{-1}$.

The partition of unity is given as above but for more flexibility we introduce a variable interval width $\tau$, { and write
\[
\varphi(x) = \chi_\tau * \chi_\tau (x),
\]
where
\[
\chi_\tau(x) =\begin{cases}
1&\text{if  }x\in[-\tau/2,\tau/2]\\
0 &\text{elsewhere},
\end{cases}
\]
so that clearly $\supp(\varphi) = [-\tau,\tau]$.}
We define the partition of unity by normalizing $\varphi$, namely
\[
\phi_\tau(x) := \frac{1}{\tau}\varphi(x) = \frac{1}{\tau} \chi_\tau * \chi_\tau (x).
\]
Let now
\[
\phi_k(x) := \phi_\tau(x-g_k).
\]
Then, upon setting $g_k:=\tau k$, we have
\[
\sum_{k\in\Z}\phi_k(x) \equiv 1~
\]
and thus, the family $\{\phi_k\}_{k\in\II}$ forms a partition of unity. It  follows that
\[
\cF(\phi_k)(\xi) =  \cF(\phi_\tau)(\xi){\rm e}^{-2\pi i g_k\xi} = \tau \sinc^2(\tau\xi){\rm e}^{-2\pi i g_k\xi}.
\]
We with help of Theorem \ref{littlehelp} we have
\[
J_\phi F(b) = \sum_{k\in\II} \langle F, \phi_k\rangle L_{g_k} K (b)= 
\sum_{k\in\II} \langle F, \phi_k\rangle K(b - g_k) = 
\sum_{k\in\II} \langle F, \phi_k\rangle\, 2\omega\,\sinc(2\omega\, (b- g_k))~.
\]
The choice of the partition of unity (grid density) has a significant impact on { the boundedness of $J_\Phi^{-1}$}. We pick  
$\tau=\frac{1}{2\omega}$, which implies  $g_k=\frac{k}{2\omega}$. { Notice that a grid with e.g. $\tau=1/\omega$ would not yield the desired result, as  a consequence of the Poisson summation  formula}. Our specific choice leads to
\[
\widehat{\phi_k}(\xi)= \frac{1}{2\omega}\sinc^2\left(\frac{1}{2\omega}\xi\right){\rm e}^{-2\pi i \frac{k}{2\omega}\xi}
\]
and therefore, as 
$e_k(\xi):=\frac{1}{\sqrt{4\omega}}{\rm e}^{-2\pi i \frac{k}{2\omega}\xi}$ forms
 an orthogonal basis for $L_2(\Omega)$, { which is actually not normalized because $\|e_k\|^2_{L_2(\Omega)}=1/2$}, we obtain
\[
2\omega\langle F,\phi_k\rangle = \int_{-\omega}^\omega 
\underbrace{\hat F (\xi)\sinc^2\left(\frac{1}{2\omega}\xi\right)}_{=:\hat{\tilde F}(\xi)}
{{\rm e}^{2\pi i \frac{k}{2\omega}\xi}{\rm d}\xi} = 
 \tilde F\left(\frac{k}{2\omega}\right).
 \]
It follows, similarly to the Shannon-Nyquist sampling theorem, that 
\[
J_\phi F(b) =\sum_{{k\in\Z}} \tilde F\left(\frac{k}{2\omega}\right)\,\sinc\left(2\omega\, \left(b- \frac{k}{2\omega}\right)\right) = 2 \tilde F(b) = 4\omega\,( F*\phi_{\frac{1}{2\omega}})(b)
\]
or, equivalently,
\[
\widehat{J_\Phi F} = 4\omega\, \hat F \,\widehat{\phi_{\frac{1}{2\omega}}}.
\]
As $F\in B_\Omega^2$ we have that $\tilde F\in B_\Omega^2$ and hence  $J_\Phi F \in B_{\Omega}^2$, and therefore, 
\begin{equation}
\hat F = \frac{1}{4\omega}\left(\chi_{[-\omega,\omega]}|\widehat{\phi_{\frac{1}{2\omega}}}|\right)^{-2}
\overline{\widehat{\phi_{\frac{1}{2\omega}}}}\,\,\widehat{J_\Phi F}~,\label{Fourierofleftinverse}
\end{equation}
which exists since $|\widehat{\phi_{\frac{1}{2\omega}}}|^2$ has its zeros outside $[-\omega,\omega]$. ({ Observe that for this argument the denser sampling is necessary}). Due to \eqref{Fourierofleftinverse} the (left/right) inverse $J_\Phi^{-1}$  of some $G\in B_\Omega^2$  can be expressed as a convolution of 
the form
\begin{eqnarray}
J_\Phi^{-1}G &=& {{\cF}^{-1}}
\left(\frac{1}{4\omega}\left(\chi_{[-\omega,\omega]}|\widehat{\phi_{\frac{1}{2\omega}}}|\right)^{-2}
\overline{\widehat{\phi_{\frac{1}{2\omega}}}}\right)*G\\
&=&  \frac{1}{4\omega}{{\cF}^{-1}}
\left(\left(\chi_{[-\omega,\omega]}|\widehat{\phi_{\frac{1}{2\omega}}}|\right)^{-2}\right)\,*\,\overline{\phi_{\frac{1}{2\omega}}(-\,\cdot)}*G~.
\end{eqnarray}
Applying the generalized Young inequality \eqref{generalizedYoung} yields for $1+1/r = 1/p+1/q$
and $1+1/p = 1/t+1/s$
\begin{eqnarray}
\|J_\Phi^{-1}G\|_{L_r} 
&\leq&  \frac{1}{4\omega}\left\|{\cF}^{-1}
\left(\left(\chi_{[-\omega,\omega]}|\widehat{\phi_{\frac{1}{2\omega}}}|\right)^{-2}\right)\right\|_{L_t}\,
\left\|\phi_{\frac{1}{2\omega}}\right\|_{L_s}\, \|G\|_{L_q}~.
\end{eqnarray}
We observe by the support  properties  of $\phi_{\frac{1}{2\omega}}$ that 
\[
\left\|\phi_{\frac{1}{2\omega}}\right\|_{L_s}\leq (4\omega)^{1/s}
\]
and if $t>1$ that (see Appendix \ref{boundedness})
\begin{equation}
\left\|{\cF}^{-1}
\left(\left(\chi_{[-\omega,\omega]}|\widehat{\phi_{\frac{1}{2\omega}}}|\right)^{-2}\right)\right\|_{L_t}\leq C<\infty~.\label{dickeresBrett}
\end{equation}
Therefore, we finally obtain
\[
\|J_\Phi^{-1}G\|_{L_r} \leq \frac{ (4\omega)^{1/s}}{4\omega}\,C\,\|G\|_{L_q}
\]
for which we have the relation $2+1/r = 1/t+1/s+1/q$. As $\phi_{\frac{1}{2\omega}}$ belongs also to $ L_1$ it can be simplyfied to $1+1/r = 1/t+1/q$ or equivalenty $r = tq/(t+q-tq)$. For each $r>1$ we find some $q>1$ while fulfilling $t>1$. Setting $t=q=1+\varepsilon$, we have $r=(1+\varepsilon)/(1-\varepsilon)>1$ while ensuring $0<\varepsilon<1$ and we obtain $\varepsilon = (r-1)/(r+1)$ and hence $q = 2r/(r+1)>1$.

 Now,  we  obtain a first  intermediate result:  the existence and continuity of the left inverse imply by  Theorem \ref{mainresult} 
 the existence of an  expansion of the form \eqref{hurra} .

  \subsubsection{Atomic Decomposition}
 The next step is to establish  the atomic decomposition property.  Due to Theorem  \ref{mainresult} , we have to show that the 
 $L_p$-norms of the $Q$-oscillation of $J_{\phi} ^{-1}K$  are finite. To this end, we prove that this property holds for {\em all}
  elements in the reproducing kernel  space.  We start with some useful observations. 
\begin{itemize}
\item[i)] 
Let $F\in{\cM^Y}$, then $F\in C^\infty$, since
\begin{equation} \label{useful1}
{\frac{{\rm d}^n}{{\rm d} x^n}F= \frac{{\rm d}^n}{{\rm d} x^n}(F*K)=F*\frac{{\rm d}^n}{{\rm d} x^n}K}
\end{equation}
\item[ii)]
 In Appendix  \ref{appendix2} , we show that  the Shannon kernel satisfies
\begin{equation} \label{useful2}
{\frac{{\rm d}^n}{{\rm d} x^n}}K\in L_p~,~~~~~\text{for all} \quad p>1~.
\end{equation}
\item[iii)]
 Then ${\frac{{\rm d}^n}{{\rm d} x^n} F\in \cM^Y}$, since 
\[
{\frac{{\rm d}^n}{{\rm d} x^n }F=\frac{{\rm d}^n}{{\rm d} x^n} (F*K) 
= \left(\frac{{\rm d}^n}{{\rm d} x^n} F\right) * K},
\]
and, , by the Young  inequality, 
\begin{equation} \label{useful3}
\left\|{\frac{{\rm d}^n}{{\rm d} x^n} }F\right\|_{L_r}=\left\|F*{\frac{{\rm d}^n}{{\rm d} x^n} }K\right\|_{L_r}
\lesssim \|F\|_{L_p}\left\|{\frac{{\rm d}^n}{{\rm d} x^n} }K\right\|_{L_q}, \quad 1/p+1/q=1+1/r,
\end{equation}
 and  for any $r>1$, we find $p,q$ such that $1/p+1/q=1+1/r.$
\end{itemize}
Therefore,  by combining   \eqref{useful1},  \eqref{useful2} and   \eqref{useful3}, we obtain
\begin{equation} \label{useful4}
{\cM^Y\subset\bigcap_{p>1}\bigcap_{k\geq0} W^k_p}  
\end{equation} 
where  {$W^k_p$ denotes the $L_p$-Sobolev space of  smoothness $k$}. 
We know that 
\[
\|\text{osc}_Q(F)\|_{L_p}<\infty
\]
if 
\[
F\in {\cM_Q^\rho(L_{p,w})}=\left\{ F\in L_0 (G)~\vert~~{\|F\|_{L_{\infty}(Q\cdot )}}\in L_{p,w}\right\}~,
\]
 see \cite{DDDSSTV} p.86, Lemma 4.3. (fortunately, it is not needed that  $\|\text{osc}_Q(F)\|_{L_p}$  is small, it only has to be
  finite).
We need to show ${\|F\|_{L_{\infty}(Q\cdot )}}\in L_{p}$.  To this end,  we use the 
Sobolev embedding  theorem as illustrated, in the DeVore-Triebel diagram:
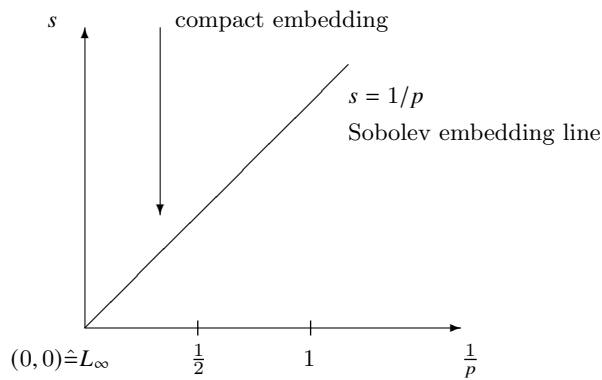
\begin{center}
\setlength{\unitlength}{1mm}
\begin{picture}(50,50)
\put(0,0){\vector(0,1){40}}
\put(-5,40){$s$}
\put(0,0){\vector(1,0){50}}
\put(50,-5){$\frac{1}{p}$}
\drawline(0,0)(35,35)
\put(35,30){$s=1/p$}
\put(35,25){Sobolev embedding line}
\put(-10,-5){$(0,0)\hat{=} L_\infty$}
\drawline(15,1)(15,-1)
\put(14,-5){$\frac{1}{2}$}
\drawline(30,1)(30,-1)
\put(29,-5){$1$}
\put(10,40){\vector(0,-1){25}}
\put(12,40){compact embedding}
\end{picture}
\end{center}
\ \\[1cm]
Here,  the point $(1/p,s)$ corresponds to the Besov space $B_{p,p}^s$, where $B_{p,p}^s=W_p^s$ for $s\not\in \NN$.
We see that for all $p>1$ the spaces $W^2_p$ embed into $L_\infty$. Therefore we get for any function $F$ in the reproducing kernel space
{
\begin{eqnarray}
\int\left(\|F\|_{L_{\infty}(Q\cdot x)}\right)^p {\rm d}x &\lesssim&
\int\left(\|F\|_{W^2_p(Q\cdot x)}\right)^p {\rm d}x\\
&\lesssim& \int\sum_{k=0}^2\left\|\frac{{\rm d}^k}{{\rm d}x^k}F\right\|^p_{L_p(Q\cdot x)} {\rm d}x \\
&=& \sum_{k=0}^2\int\int_{Q\cdot x}\left|\frac{{\rm d}^k}{{\rm d}x^k}F\right|^p {\rm d}u {\rm d}x \\
&=& \sum_{k=0}^2\int\int_{Q}\left|\frac{{\rm d}^k}{{\rm d}x^k}F(u+x)\right|^p {\rm d}u {\rm d}x\\
&=& \sum_{k=0}^2\int_{Q}\int_{\RR}\left|\frac{{\rm d}^k}{{\rm d}x^k}F(u+x)\right|^p {\rm d}x {\rm d}u\\
&=& \int_{Q} \|F\|^p_{W^2_p} {\rm d}u = \|F\|^p_{W^2_p}  \int_{Q} {\rm d}u <\infty~.
\end{eqnarray}
}
This is true for all ${F\in\cM^Y}$. But $J^{-1}_\phi K$ is contained in ${\cM^Y}$ and therefore,
$$\|{\rm osc}_Q (J_\phi^{-1}K)\|_{L_{q'}}<\infty,$$
and we are done.

{
\begin{remark} \label{ discussion4}
As mentioned earlier, we are now in a position to show how to handle the case $w=m$.   We know that
\[
 \norm{{\rm osc}_Q(F)}_{L_{p,w}}<\infty
 \]
 if $F\in\cM_Q^\rho(L_{p,w})$.
Using as before the Sobolev embedding theorem, we get
\begin{eqnarray*}
\left\| \| F\|_{L_{\infty}(Q\cdot )}\right\|_{L_{p,w}}&=&\int_\RR\| F\|_{L_{\infty}(Q\cdot x)}^p w(x)^p {\rm d}x\\
&\lesssim& \sum_{n=0}^2 \int_\RR\int_Q \left| \frac{{\rm d}^n }{{\rm d}x^n}F(u+x)\right|^p {\rm d}u\, w(x)^p {\rm d}x\\
&\lesssim& \sum_{n=0}^2 \int_Q\int_\RR \left| \frac{{\rm d}^n }{{\rm d}x^n}F(u+x)\right|^p  w(x)^p {\rm d}x\, {\rm d}u\\
&=& \sum_{n=0}^2 \int_Q\int_\RR \left| \frac{{\rm d}^n }{{\rm d}x^n}F(x')\right|^p  w(x'-u)^p {\rm d}x'\, {\rm d}u\\
&\lesssim& \int_Q w(-u)^p {\rm d}u \sum_{n=0}^2\int_\RR \left| \frac{{\rm d}^n }{{\rm d}x^n}F(x')\right|^p w(x')^p   {\rm d}x'\\
&\leq& \int_Q w(-u)^p {\rm d}u\,\, \|F\|^p_{W^2_{p,w}}.
\end{eqnarray*}
The weighted Young inequality, applied to  the case $m=w$,  implies that
$$\|F\|^p_{W^2_{p,w}}<\infty ~~~\forall F\in \cM^Y~.$$
Since $J^{-1}_\phi K \in \cM^Y$, the result follows.
 \end{remark} 
 }
\subsection{Modulation Spaces}\label{sec:modulation}
Modulation spaces are based on the translation and modulation operators on functions, { namely}
$$
T_a f(t) = f(t-a),\quad M_bf(t) = \mathrm{e}^{2\pi \i bt} f(t)
$$
with corresponding Fourier transforms
$$
\widehat{T_a f}   = M_{-a} \hat f, \quad 
\widehat{M_b f}   = \hat f (T_b \cdot) .
$$
The \emph{reduced Heisenberg group} $\mathbb H_r$ is the locally compact { group}
 $\mathbb H_r = \R^{2d} \times \mathbb T$ 
with multiplication and inversion
\[
(x,\omega,z) (x',\omega',z') = (x+x',\omega + \omega', z z' \mathrm{e}^{i \pi(x' \omega - x \omega')}),
\quad 
(x,\omega,z)^{-1} = (-x,-\omega,\bar z).
\]
The {(group)} convolution of  {$f,g\in L_1(\mathbb H_r)$} is given by
\begin{align*}
(f * g) (x,\omega,z) 
&=
\int_{ \mathbb H_r}
f(x',\omega',z')g\left((x',\omega',z')^{-1} (x,\omega,z) \right) \d x'\d \omega' \d z' 
\\
&=
\int_{ \mathbb H_r}
f(x',\omega',z') g\left(x-x',\omega - \omega',z \overline{z'} \mathrm{e}^{\pi \i (x' \omega - \omega'x) } \right) \d x'\d \omega' \d z'.
\end{align*}
Introducing the { mapping $j\colon{\mathbb C}^{\R^2}\to{\mathbb C}^{\mathbb H_r}$ 
defined on $F\colon\R^2\to{\mathbb C}$ by}
\[
(j F)(x,\omega,z) = \bar z \mathrm{e}^{\pi \i x\omega} F(x,\omega),
\]
{ which induces an isometry from $L_2(\R^2)$ into $L_2(\mathbb H_r)$ 
the convolution can be rewritten as}
\begin{align}
(j F * jG) (x,\omega,z) 
&=
\int_{ \mathbb H_r}
\overline{z'}
\mathrm{e}^{ \pi \i x' \omega'} F(x',\omega') G(x-x', \omega - \omega') 
\bar z z' \mathrm{e}^{- \pi \i( x' \omega - \omega'x)}
\mathrm{e}^{\pi \i (x-x') (\omega - \omega')}  \d x'\d \omega' \d z' \, 
\\
&=
\mathrm{e}^{ \pi \i x \omega} \bar z 
\int_{\R^2} F(x',\omega') G (x-x', \omega - \omega')
\mathrm{e}^{2\pi \i x' (\omega' - \omega) } \d x'\d \omega' 
\\
&=
j(F \odot G) (x,\omega,z)
\end{align}
where {  $F \odot G$ is given, for $F,G\in L_1({\mathbb R}^2)$, by}:
\[
(F \odot G) (x,\omega) = \int_{\R^2} F(x',\omega') G(x-x', \omega - \omega')
\mathrm{e}^{2\pi \i x'( \omega' -  \omega) } \d x' \d \omega'.
\]
{ Therefore { the mapping $j$ intertwines the group convolution on $\mathbb H_r$  with $\odot$, that is}
\[
jF*jG=j(F \odot G)
\]
for, say, $F,G\in L_1({\mathbb R}^2)$.} The Schr\"odinger representation $\pi: \mathbb H_r \to \mathcal U(L_2(\R))$ is the unitary representation given by
$$
(\pi_{(x,\omega,z)} f) (t) = z \mathrm{e}^{-\pi \i x \omega} \mathrm{e}^{2 \pi \i t \omega} f(t-x).
$$
The corresponding voice transform is
\begin{align}
{ V_g(f) (x,\omega,z)} 
&= \langle f, \pi_{(x,\omega,z)}g \rangle 
= 
\int_{\R} f(t)  \overline{z  \mathrm{e}^{-\pi \i x \omega}  \mathrm{e}^{2\pi \i \omega t} {g(t-x)}}\,\d t\\
&=\bar z \mathrm{e}^{\pi \i x \omega} \int_{\R}  f(t)  \overline{{\mathrm{e}^{2\pi \i \omega t}g(t-x)}}\,\d t \\
&=j ({U_g f) (x,\omega,z)},
\end{align}
where
\[
{U_g(f) (x,\omega)} = \int_{\R} f(t) \overline{{\mathrm{e}^{2\pi \i \omega t}g(t-x)}}\, \d t 
={\langle f,M_\omega T_x{g}\rangle}.
\]
For the function $g =\chi_{[-\frac12,\frac12]}={\check{g}}$, 
we consider the reproducing kernel
$$
K(x,\omega) = {U_g g(x,\omega)} = \int_{\R} \underbrace{g(t) \overline{g(t-x)} }_{h_x(t)}
\mathrm{e}^{- 2\pi \i t \omega }\, \d t 
= { \widehat{h_x}(\omega)}.
$$
Its Fourier transform is given by
\begin{align} 
\hat K (\xi, \eta) 
&= \int_{\R^2}  \hat h_x(\omega) 
\mathrm{e}^{-2 \pi \i x \xi }\mathrm{e}^{-2 \pi \i \eta \omega }\, \d x \d \omega
= 
\int_{\R} h_x(-\eta) \mathrm{e}^{-2 \pi \i x \xi }\,\d x
\\
&=
\int_{\R}  g (-\eta) \overline{g(-\eta - x) }\mathrm{e}^{-2 \pi \i x \xi }\,\d x 
= 
g (-\eta) \int_{\R}  \overline{ g(x -\eta) }\mathrm{e}^{2 \pi \i x \xi }\, \d x
\\
&=
g (-\eta) \, \widehat{{\overline g}}(\xi) \, \mathrm{e}^{2 \pi \i \eta \xi } 
\\
&=
\chi_{[-\frac12,\frac12]}(\eta) \, \mathrm{sinc} (\xi) \, \mathrm{e}^{2\pi \i \eta \xi}. \label{kernel}
\end{align}
We have that $K \in L_p(\R^2)$ { for all} $p >1$, but $K \notin L_1(\R^2)$.

\begin{remark}
  From a practical point of view, our choice of the window function g
  might not be optimal. It possesses a very good time localization,
  but the frequency localization is bad. Very often, one strives fo a
  compromize and uses, e.g., the Gaussian window.  Nevertheless, as a
  test example for our theory, our choice is fine.
\end{remark}

Let
$$\mathcal M = \{F \in L_2(\R^2): F \odot K = F\}.$$

We will need the Fourier transform of functions in $\mathcal M$.

\begin{lemma}\label{lem1}
For every $F\in\mathcal M$ it holds that
\[
\widehat{F \odot K} (\xi, \eta) 
= \mathrm{sinc} (\xi) \mathrm{e}^{2\pi \i \xi \eta} \int_{\R}\hat F(\xi',\eta) \,\mathrm{sinc}( \xi') \mathrm{e}^{-2\pi \i \xi' \eta }\,\d \xi' .
\]
\end{lemma}

\begin{proof}
Since $F \in \mathcal M$, we have 
\begin{align}
\hat F(\xi,\eta) &= \widehat{F \odot K} (\xi, \eta) 
\\
&=
\int_{\R^2} \int_{\R^2}  F(x',\omega') K(x-x',\omega - \omega')  \mathrm{e}^{2\pi \i x'(\omega' - \omega)}
 \mathrm{e}^{-2\pi \i x\xi}  \mathrm{e}^{-2\pi \i \omega \eta}\, \d x \d \omega \d x' \, \d \omega' 
\\
&=
\int_{\R^2}  \Bigl( 
\int_{\R^2} K(x-x',\omega - \omega')  \mathrm{e}^{-2\pi \i x' \omega}
 \mathrm{e}^{-2\pi \i x\xi}  \mathrm{e}^{-2\pi \i \omega \eta}\,\d x \d \omega\Bigr)\mathrm{e}^{2\pi \i x' \omega'} F(x',\omega')\,\d x' \d \omega' \,
\\
&=
\int_{\R^2}  
 \Bigl(\int_{\R^2}  K(x,\omega)  
  \mathrm{e}^{-2\pi \i  x \xi}
  \mathrm{e}^{-2\pi \i  \omega(x' + \eta)}\,\d x \d \omega\Bigr)
  F(x',\omega')
 \mathrm{e}^{-2\pi \i x' \xi}
 \mathrm{e}^{-2\pi \i \omega' \eta}\, \d x' \d \omega'
 \\
&= \int_{\R^2}   \hat K(\xi, x'+ \eta)F(x',\omega')
 \mathrm{e}^{-2\pi \i x' \xi}
 \mathrm{e}^{-2\pi \i \omega' \eta}\, \d x' \d \omega'. 
\end{align}
{ Using now  \eqref{kernel}, we obtain}
\begin{align}
\hat F(\xi,\eta) 
&=
\mathrm{sinc} (\xi) \mathrm{e}^{2\pi \i \xi \eta } 
\int_{\R^2} F(x',\omega') \chi_{[-\frac12,\frac12]} (x' + \eta) \mathrm{e}^{-2\pi \i  \eta \omega'}\,\d x' \d \omega' 
\\
&=
\mathrm{sinc} (\xi) \mathrm{e}^{2\pi \i \xi \eta } 
\int_{\R^2} F(x'- \eta,\omega') \chi_{[-\frac12,\frac12]} (x') \mathrm{e}^{-2\pi \i  \eta \omega'}\, \d x' \d \omega' 
\\
&=
\mathrm{sinc} (\xi) \mathrm{e}^{2\pi \i \xi \eta } 
\int_{\R}   \mathcal F_\omega F(x' - \eta,\eta) \chi_{[-\frac12,\frac12]} (x') \d x'
\\
&=
\mathrm{sinc} (\xi) \mathrm{e}^{2\pi \i \xi \eta} 
\int_{\R^2}  
\chi_{[-\frac12,\frac12]} (x') \, \mathrm{e}^{2\pi \i \xi' (x'- \eta)} \hat F(\xi',\eta)\,\d x' d\xi' 
\\
&=
\mathrm{sinc} (\xi) \mathrm{e}^{2\pi \i \xi \eta}
\int_{\R} \hat F(\xi',\eta) \,\mathrm{sinc}( \xi') \mathrm{e}^{-2\pi \i \xi' \eta }\, \d \xi',
\end{align}
{ as desired}.
\end{proof}

\begin{remark}
  There is a remarkable difference from the results  for the Shannon case.     There, the resulting coorbit spaces according to Theorem 4 could be identified 
as intersections of well-known spaces, namely the Paley-Wiener spaces. Here, the coorbit spaces are  {\em not}   intersections of classical modulation spaces, simply because
for the non-integrable case these spaces do not exist. We end up with a really  new class of spaces. 
\end{remark}

\subsubsection{Atomic decomposition}\label{sec:modulation1}
We are finally in a position to prove the injectivity of $J_\phi$, as
well as the $L_p$--boundedness of the $Q$-oscillation for $p>1$. The
continuity of the left inverse  will be the topic of  future research.

We choose $Q = (-\frac12,\frac12)^2$ and $\mathbb Z^2$ as $Q$-dense set { and write}
\[
\phi(x,\omega) = \chi_{[-\frac12,\frac12]}(x) \chi_{[-\frac12,\frac12]}(\omega), 
\quad
\phi_{k,l}(x,\omega) = \phi(x-k,\omega-l).
\]

\begin{proposition}
The operator $J_\phi\colon \mathcal M \to \mathcal M$ defined by 
$$
J_\phi (F) = \sum_{k,l \in \Z} \langle F,\phi_{k,l} \rangle K(x+k,\omega+ l)
$$
is injective.
\end{proposition}

\begin{proof}
We show that $J_\phi (F) = 0$ if and only if $\langle F,\phi_{k,l} \rangle = 0$
for all $k,l \in \mathbb Z$.
Using Plancherel's theorem, we obtain
\begin{align}
0 &= \langle F, \phi_{k,l} \rangle
= 
\langle \widehat{F \odot K}, \hat \phi_{k,l} \rangle 
\\
&= 
\int_{\R^3}\mathrm{sinc} (\xi) \mathrm{e}^{2\pi \i \xi \eta }
\hat F(\xi',\eta) \mathrm{sinc} (\xi') \mathrm{e}^{-2\pi \i \xi' \eta } \, 
\mathrm{sinc} (\xi) \mathrm{e}^{-2\pi \i k \xi  }
\mathrm{sinc} (\eta) \mathrm{e}^{-2\pi \i l \eta  }\, \d \xi \d \eta \, \d \xi'.
\end{align}
{ Observe that}
\[
\int_{\R}  \mathrm{sinc} ^2 (\xi) \mathrm{e}^{2\pi \i \xi (\eta -k)}\, \d \xi 
= M_2(\eta - k),
\]
where $M_2$ denotes the centered cardinal $B$-spline of order 2.
Thus, we get
\begin{align}
0
&= \int_{\R}M_2(\eta - k)  
\mathrm{e}^{-2\pi \i l \eta } \; \mathrm{sinc} (\eta)
\Bigl(\int_{\R}\hat F(\xi',\eta) 
\mathrm{sinc} (\xi') \mathrm{e}^{-2\pi \i \xi' \eta }\,\d\xi' \Bigr)\, \d \eta.
\end{align}
Now $\{M_2(\eta - k) \mathrm{e}^{-2\pi \i l \eta }:k,l \in \mathbb Z\}$ 
is a Gabor frame in $L_2(\R)$ so that the above implies
\begin{align}
0
&= 
\mathrm{sinc} (\eta) \int_{\R}  \hat F(\xi',\eta) 
\mathrm{sinc} (\xi') \mathrm{e}^{-2\pi \i \xi' \eta }\,\d \xi'
\end{align}
and then
\[
0=\int_{\R} \hat F(\xi',\eta) 
\mathrm{sinc} (\xi') \mathrm{e}^{-2\pi \i \xi' \eta}\d \xi'  \quad \eta-a.e.
\]
Consequently, by Lemma \ref{lem1}, we have $\hat F(\xi,\eta) = 0$
and then $F(x,\omega) = 0$ a.e..
\end{proof}

The next step is to show the $L_p$-boundedness of the
$Q$-oscillation. We will follow closely the argument in Section.5.1.4,
i.e., we show the boundedness  for all functions in the reproducing
kernel space by means of Sobolev embeddings. We first observe
$$\frac{\partial K}{\partial\omega},\frac{\partial K}{\partial x}\in L_p(\RR^2)
	\quad\text{ for all }p>1.$$
This becomes clear when writing a more explicit expression for $K$. We have
$$
	h_x(t)=\chi_{[-\frac{1}{2},\frac{1}{2}]}(t)\chi_{[x-\frac{1}{2},x+\frac{1}{2}]}(t)
		=\begin{cases}
			0, &|x|>1\,,\\
			\chi_{[-\frac{1}{2},x+\frac{1}{2}]}(t), &-1\leq x<0,\\
			\chi_{[x-\frac{1}{2},\frac{1}{2}]}(t), &0\leq x\leq 1.
		\end{cases}
$$
From
$$
	\widehat\chi_{[a,b]}(\omega)=(b-a){\rm e}^{2\pi i\frac{b+a}{2}\omega}\sinc\bigl((b-a)\omega\bigr)
$$
this results in the explicit expression
\begin{align*}
	K(x,\omega)=\widehat{ h_x}(\omega)
		&=\begin{cases}
			0, &|x|>1\\
			{\rm e}^{2\pi i\frac{x}{2}\omega}\sinc\bigl((1+x)\omega\bigr), &-1\leq x<0\\
			{\rm e}^{2\pi i\frac{x}{2}\omega}\sinc\bigl((1-x)\omega\bigr), &0\leq x\leq 1
			\end{cases}\\
		&={\rm e}^{\pi i x\omega}\sinc\bigl((1-|x|)\omega\bigr)\chi_{[-1,1]}(x).
\end{align*}
Thus as a function of $\omega$ the kernel $K$ is a smooth function, whereas w.r.t. $x$ it is only Lipschitz. However, since $K$ and its weak derivatives are compactly supported w.r.t. $x$, this suffices to belong to $L_p$.

\begin{lemma}
	Young's inequality transfers to the $\odot$-product, i.e. we have
	$$\|F\odot G\|_{L_r}\lesssim\|F\|_{L_p}\|G\|_{L_q}$$
	for all $F\in L_p(\R^2)$ and $G\in L_q(\R^2)$, whenever $1/p+1/q=1+1/r$.
\end{lemma}

\begin{proof}
	This follows immediately from the intertwining relation
	$$j(F\odot G)=jF\ast jG$$
	for the isometry $j:L_p(\RR^2)\to L_p(\mathbb{H}_r)$ and Young's inequality for the group
	$\mathbb{H}_r$. More precisely
	\begin{align*}
		\|F\odot G\|_{L_r(\RR^2)}
			&=\|j^{-1}(jF\ast jG)\|_{L_r(\RR^2)}
				=\|jF\ast jG\|_{L_r(\mathbb{H}_r)}\\
			&\lesssim\|jF\|_{L_p(\mathbb{H}_r)}\|jG\|_{L_q(\mathbb{H}_r)}
				=\|F\|_{L_p(\RR^2)}\|G\|_{L_q(\RR^2)}\,.
	\end{align*}
\end{proof}

\begin{lemma}
	For $F\in\mathcal{M}$ we have
	$$\frac{\partial F}{\partial x}=F\odot\frac{\partial K}{\partial x}\in L_p(\R^2)$$
	for all $p>1$ as well as
	$$\frac{\partial F}{\partial\omega}=\frac{\partial F}{\partial\omega}\odot K\,,$$
	and consequently $\frac{\partial F}{\partial\omega}\in\mathcal{M}$.
\end{lemma}

\begin{proof}
	The first identity follows immediately from the reproducing
        property of $K$ 
\[ \frac{\partial}{\partial x} (F\odot G)= F\odot \frac{\partial G}{\partial x}\]
together with Young's
	inequality. The second identity follows from observing
	$$F\odot K(x,\omega)=\int F(x-x',\omega-\omega')K(x',\omega')
		e^{2\pi i (x'-x)\omega'}{\rm d}x'{\rm d}\omega',$$
	and once more Young's inequality.
\end{proof}

\begin{corollary}
	For $F\in\mathcal M$ we also have $\frac{\partial^2 F}{\partial x\partial\omega}\in L_p(\R^2)$ for all
	$p>1$.
\end{corollary}

In short: While we no longer have the reproducing property for the $x$-derivatives, we nevertheless retain the needed integrability properties. Those results can be combined into

\begin{corollary}
	The space $\mathcal M$ embeds into $S^1_p W(\RR^2)$ for all $p>1$, the first order Sobolev space of
	dominating mixed smoothness.
\end{corollary}

For spaces of dominating mixed smoothness, the counterpart of the classical Sobolev embedding can be formulated as
\[
	S^s_p W(\RR^d)\hookrightarrow L_\infty(\RR^d)\quad\text{provided}\quad s>1/p\,,
\]
see, e.g., \cite{ST}, Chapter 2.4  for details.  Note that the condition is the same as for the univariate embedding. Ultimately, we now conclude
\begin{align*}
	\int\left(\|F\|_{L_{\infty}(Q\cdot(x,\omega))}\right)^p {\rm d}x{\rm d}\omega
		&\lesssim
			\int\left(\|F\|_{S^1_p W(Q\cdot(x,\omega))}\right)^p {\rm d}x{\rm d}\omega\\
		&\lesssim
			\int\left(\left\|\frac{\partial F}{\partial x}\right\|^p_{L_p(Q\cdot(x,\omega))}
				+\left\|\frac{\partial F}{\partial\omega}\right\|^p_{L_p(Q\cdot(x,\omega))}
				+\left\|\frac{\partial^2 F}{\partial x\partial\omega}\right\|^p_{L_p(Q\cdot(x,\omega))}
				\right){\rm d}x{\rm d}\omega \\
		&=\int\int_{Q\cdot(x,\omega)}\left|\frac{\partial F}{\partial x}(u,\eta)\right|^p
			{\rm d}u {\rm d}\eta {\rm d}x{\rm d}\omega+\cdots \\
		&=\int\int_{Q}\left|\frac{\partial F}{\partial x}(u+x,\eta+\omega)\right|^p
			{\rm d}u {\rm d}\eta {\rm d}x{\rm d}\omega+\cdots \\
		&=\int_{Q}\int_{\RR^2}\left|\frac{\partial F}{\partial x}F(u+x,\eta+\omega)\right|^p
			{\rm d}x{\rm d}\omega 	{\rm d}u {\rm d}\eta \\
		&=\int_{Q}\|F\|^p_{S^1_p W(\RR^2)}	{\rm d}u {\rm d}\eta 
			=\|F\|^p_{S^1_p W(\RR^2)} 	{\rm d}u {\rm d}\eta 
\int_{Q} 	{\rm d}u {\rm d}\eta <\infty~.
\end{align*}


\section{Appendix}

\subsection{Generalized Weighted Young Inequality} \label{appendix1}

\begin{lemma}\label{LemmageneralizedYoung}
Let $w$ be the control weight and suppose that $m$ is $w$-moderate. Then 
\begin{equation}\label{generalizedYoung}
\norm{H*F}_{L_{r,m}}\lesssim \norm{H}_{L_{p,m}}\norm{F}_{L_{q,w}},~~~ \text{where}~~~~~~~1+\frac{1}{r}=\frac{1}{q}+\frac{1}{p}~.
\end{equation}
\end{lemma}
\begin{proof}
The $w$-moderateness of $m$ implies
$$m(h) = m(gg^{-1}h) \lesssim m(g)w(g^{-1}h)=m(g)w(h^{-1}g)~.$$
Therefore, it follows
\begin{eqnarray*}
\norm{H*F}_{L_{r,m}} &=& \left(\int_G |(H*F)(h)|^r m(h)^r \d h\right)^{1/r}\\
&=&\left(\int_G \left|\int_G H(g)F(h^{-1}g) \d g\right|^r m(h)^r\d h\right)^{1/r}\\
&=&\left(\int_G \left|\int_G H(g)F(h^{-1}g) m(h) \d g\right|^r  \d h\right)^{1/r}\\
&\lesssim&  \left(\int_G \left|\int_G w(h^{-1}g)F(h^{-1}g)H(g) m(g) \d g\right|^r  \d h\right)^{1/r}\\
&=& \left(\int_G |\left((m\cdot H)*(w\cdot F)\right)(h)|^r  \d h\right)^{1/r}\\
&\lesssim& \norm{w\cdot F}_{L_{q}} \norm{m\cdot H}_{L_{p}}
= \norm{H}_{L_{p,m}} \norm{F}_{L_{q,w}}.
\end{eqnarray*}
\end{proof}

\subsection{$L_p$ Estimates of the Shannon Kernel and its Derivatives} \label{appendix2}
{ In this section we prove that the all the  derivatives of the  Shannon kernel 
$x\mapsto 2\omega\sinc(2\omega x)$ are in all $L_p$-spaces, for $p>1$. For simplicity, we write
\[
K(x) = \frac{\sin x}{x}
\]
and prove the statement for $K$, which is obviously equivalent.}
We show by induction that
\begin{equation}\label{eq_diff_sinc}
\frac{\d ^n}{\d x^n}K(x) = x^{-2^n}\sum_{k=0}^{2^n-1}x^k p_{k,n}(x)~,
\end{equation}
where the $p_{k,n}$ are  trigonometric polynomials. { The case $n=1$ is true because} 
\[
\frac{\d }{\d x}K = \frac{x\cos x-\sin x}{x^2}.
\]
Suppose that (\ref{eq_diff_sinc}) holds for some $n$. Then
\begin{eqnarray}
\frac{\d ^{n+1}}{\d x^{n+1}}K &=& \frac{\d }{\d x}\left( \frac{\d ^n}{\d x^n}\right) = 
\frac{\d }{\d x}\left( x^{-2^n}\sum_{k=0}^{2^n-1}x^k p_{k,n}(x) \right)\\
&=& x^{-2^n}\left(\sum_{k=0}^{2^n-1}x^k p'_{k,n}(x)+\sum_{k=1}^{2^n-1}kx^{k-1} p_{k,n}(x)\right)-2^nx^{-2^n-1}\sum_{k=0}^{2^n-1}x^k p_{k,n}(x)\\
&=& \frac{x^{2^n}\left(\sum_{k=0}^{2^n-1}x^k p'_{k,n}(x)+\sum_{k=1}^{2^n-1}kx^{k-1} p_{k,n}(x)-2^n\sum_{k=0}^{2^n-1}x^{k-1} p_{k,n}(x)\right)}{\left(x^{2^n}\right)^2}\\
&=& \frac{\sum_{k=0}^{2^n-1}x^{2^n+k} p'_{k,n}(x)+\sum_{k=1}^{2^n-1}kx^{2^n+k-1} p_{k,n}(x)-2^n\sum_{k=0}^{2^n-1}x^{2^n+k-1} p_{k,n}(x)}{\left(x^{2^n}\right)^2}\\
&=& \frac{\sum_{k=0}^{2^{n+1}-1}x^{k} p_{k,n+1}(x)}{x^{2^{n+1}}}~.
\end{eqnarray}
Hence, (\ref{eq_diff_sinc}) is proved. 
Since the  { powers of $x$ that appear in  the  just derived expressions of $K^{(n)}(x)$ are $x^{-\ell}$ with $\ell\geq 1$,  it follows for } $|x|\geq C$ that
$$\left| \frac{\d ^n}{\d x^n}K\right| =\left|
 \frac{1}{x}  \sum_{k=0}^{2^n-1}x^{k-(2^n-1)} p_{k,n}(x)\right|\leq
 \left|\frac{1}{x}\right|\underbrace{\left| \sum_{k=0}^{2^n-1}x^{k-(2^n-1)} p_{k,n}(x)\right|}_{\leq C'}$$
 and therefore
 $$\left\|\frac{\d ^n}{\d x^n}K\right\|_{L_p}\leq C_p~~,~~~~p>1.$$
A yet more explicit formula holds. Indeed, from the product rule
$$(uv)^{(n)} = \sum_{k=0}^n \left(\begin{array}{c} n\\ k\end{array}\right)u^{(k)}v^{(n-k)},$$
{ choosing}  $u(x)=\sin(x)$ and $v(x) = x^{-1}$ 
it follows from $(v(x))^{(n)} = (-1)^n n! x^{-(n+1)}$ that
\[
\left(\frac{\sin x}{x}\right)^{(n)} = x^{-n} \sum_{k=0}^n \frac{n!}{k!}(-1)^{n-k}(\sin(x))^{(k)} x^{k-1}.
\]

\subsection{A Building Block for the Continuity  of $J_\phi^{-1}$}\label{boundedness}

We have to bound \eqref{dickeresBrett}. To this end, we split the $L_t$-norm as follows
\begin{eqnarray*}
J &=& \left\|\cF^{-1}
\left(\left(\chi_{[-\omega,\omega]}|\widehat{\phi_{\frac{1}{2\omega}}}|\right)^{-2}\right)\right\|^t_{L_t}
=       \underbrace{\int_{-\varepsilon}^\varepsilon | \cF^{-1} 
\left(\left(\chi_{[-\omega,\omega]}|\widehat{\phi_{\frac{1}{2\omega}}}|\right)^{-2}(x)\right)|^t \d x}_{=:J_1}                                                                                                                                            \\
&&     +   \underbrace{\int_{-\infty}^{-\varepsilon}| \cF^{-1} 
\left(\left(\chi_{[-\omega,\omega]}|\widehat{\phi_{\frac{1}{2\omega}}}|\right)^{-2}(x)\right)|^t \d x}_{=:J_2}                                                                                                                                                 +   \underbrace{\int_\varepsilon^\infty | \cF^{-1}
\left(\left(\chi_{[-\omega,\omega]}|\widehat{\phi_{\frac{1}{2\omega}}}|\right)^{-2}(x)\right)|^t \d x}_{=:J_3}                                                                                                                                                 ~.
\end{eqnarray*}
First, we consider 
$$\cF^{-1}
\left(\left(\chi_{[-\omega,\omega]}|\widehat{\phi_{\frac{1}{2\omega}}}|\right)^{-2}\right)(x) =
\int_{-\omega}^\omega |\widehat{\phi_{\frac{1}{2\omega}}}(\xi)|^{-2}\,{\rm e}^{2\pi i x\xi}\d \xi$$
and with
$$|\widehat{\phi_{\frac{1}{2\omega}}}(\xi)|^{-2}\big\vert_{[-\omega,\omega]}=|1/(2\omega)\,\sinc^2\left(\xi/(2\omega)\right)|^{-2}\big\vert_{[-\omega,\omega]}\leq (2\omega)^2\sinc^{-4}(1/2) = \frac{\omega^2 \pi^4}{4}$$
we obtain
$$J_1 \le \int_{-\varepsilon}^\varepsilon \left| \int_{-\omega}^\omega
\frac{\omega^2 \pi^4}{4}\d \xi \right|^t \d x = 2\varepsilon \left(\frac{\omega^3 \pi^4}{2}\right)^t~~.$$
To estimate $J_2$ we apply integration by parts and obtain 
\begin{eqnarray*}
\cF^{-1}
\left(\left(\chi_{[-\omega,\omega]}|\widehat{\phi_{\frac{1}{2\omega}}}|\right)^{-2}\right)(x) &=&
\int_{-\omega}^\omega \underbrace{|\widehat{\phi_{\frac{1}{2\omega}}}(\xi)|^{-2}}_{u(\xi)}
\underbrace{\,{\rm e}^{2\pi i x\xi}}_{v'(\xi)}\d \xi\\
&=& \frac{1}{2\pi i x}\left\{ |\widehat{\phi_{\frac{1}{2\omega}}}(\xi)|^{-2}\,{\rm e}^{2\pi i x\xi}\bigg\vert_{-\omega}^\omega
+ 2 \int_{-\omega}^\omega \frac{\mathfrak{Re}(\widehat{\phi_{\frac{1}{2\omega}}}'\,\overline{\widehat{\phi_{\frac{1}{2\omega}}}})}
{|\widehat{\phi_{\frac{1}{2\omega}}}|^4} 
\,{\rm e}^{2\pi i x\xi} \d \xi\right\}\\
&=& \frac{1}{2\pi i x}\left\{ \frac{\omega^2\pi^4}{2}2i\sin(2\pi \omega x)
+ 2 \int_{-\omega}^\omega 
\underbrace{\frac{\mathfrak{Re}(\widehat{\phi_{\frac{1}{2\omega}}}'\,\overline{\widehat{\phi_{\frac{1}{2\omega}}}})}
{|\widehat{\phi_{\frac{1}{2\omega}}}|^4} }_{\leq C<\infty} 
\,{\rm e}^{2\pi i x\xi} \d \xi\right\}\\
&\leq& \frac{1}{\pi |x|}
\left\{ \frac{\omega^2\pi^4}{2}
+ 2\omega C\right\}
\end{eqnarray*}
resulting in 
\begin{eqnarray*}
J_2 &\leq&  \left( \frac{\omega^2\pi^2}{2}
+ \frac{2\omega C}{\pi}\right)^t\int_{-\infty}^{- \varepsilon}\frac{1}{|x|^t} \d x = 
 \left( \frac{\omega^2\pi^2}{2}
+ \frac{2\omega C}{\pi}\right)^t \frac{1}{t-1}\frac{1}{\varepsilon^{t-1}}~.
\end{eqnarray*}
As $J_3$ is treated analougsly, we finally obtain
$$J \leq 2\varepsilon \left(\frac{\omega^3 \pi^4}{2}\right)^t + 2 
 \left( \frac{\omega^2\pi^2}{2}
+ \frac{2\omega C}{\pi}\right)^t \frac{1}{t-1}\frac{1}{\varepsilon^{t-1}} < \infty$$
and we are done.
\subsection*{Acknowledgements}
The research by EDV has been supported by the MIUR Grant PRIN
202244A7YL and  the MUR PNRR project PE0000013 CUP J53C22003010006
’Future Artificial Intelligence Research (FAIR)’. The research by EDV
and FDM has been
supported by the MUR Excellence Department Project awarded to
Dipartimento di Matematica, Universit\`a di Genova, CUP
D33C23001110001. EDV is a member of the Gruppo Nazionale per l’Analisi 
Matematica, la Probabilit\`a e le loro Applicazioni (GNAMPA) of the Istituto Nazionale di Alta Matematica (INdAM). 
GS acknowledges funding within the DFG Excellence Cluster MATH+.

\end{document}